\documentclass[12pt, english, a4paper]{amsart}

\usepackage{graphicx}

\usepackage{amsmath,amssymb}
\usepackage{bm}
\usepackage{euscript,color}
\usepackage{babel}
\usepackage{amsthm}
\usepackage{amsfonts}
\usepackage{latexsym}
\usepackage{fancyhdr}

\usepackage[left=2.5cm, top=3cm, bottom=3cm, right=2.5cm]{geometry}

\def\mymathfont{\mathbf}

\newcommand{\R}{\mymathfont{R}}
\newcommand{\C}{\mymathfont{C}}
\newcommand{\Z}{\mymathfont{Z}}

\newcommand{\Int}{\mathop{\mathrm{Int}}\nolimits}

\newcommand{\CP}{\mathop{\mymathfont{C}\mymathfont{P}}\nolimits}

\newcommand{\pt}{\mathop{\mathrm{pt}}\nolimits}

\newcommand{\cs}{\mathop{\#}\nolimits}

\makeatletter
\renewcommand{\section}{\@startsection%
{section}
{1}
{0mm}
{1.5\bigskipamount}
{0.5\bigskipamount}
{\centering\normalsize\sc}}

\renewcommand{\paragraph}{\@startsection%
{paragraph}
{4}
{0mm}
{\bigskipamount}
{-1.25ex}
{\normalsize\sl}}

\def\provedboxcontents#1{$\square$}

\makeatother

\newtheoremstyle{thm}{}{}{\slshape}{}{\scshape}{.}{0.5em}{}

\newtheoremstyle{def}{}{}{}{}{\scshape}{.}{0.5em}{}

\newtheoremstyle{rmk}{}{}{}{}{\scshape}{.}{0.5em}{}

\newtheoremstyle{claim}{}{}{}{}{\slshape}{.}{0.5em}{}

\theoremstyle{thm}
\newtheorem{newstatement}{newstatement}

\newtheorem{theorem}[newstatement]{Theorem}
\newtheorem{corollary}[newstatement]{Corollary}
\newtheorem{proposition}[newstatement]{Proposition}

\theoremstyle{def}
\newtheorem{definition}[newstatement]{Definition}

\theoremstyle{rmk}
\newtheorem{remark}[newstatement]{Remark}

\theoremstyle{claim}

\expandafter\let\expandafter\oldproof\csname\string\proof\endcsname
\let\oldendproof\endproof
\renewenvironment{proof}[1][\proofname]{%
  \oldproof[\slshape #1]%
}{\oldendproof}

\let\geq\geqslant
\let\leq\leqslant

\let\phi\varphi
\let\epsilon\varepsilon

\renewcommand{\emph}[1]{{\slshape #1}}
\renewcommand{\em}{\sl}

\title
[Non-K\"{a}hler complex structures on $\mathrm{R}^4$]
{Non-K\"{a}hler complex structures on $\bm{\R^4}$}

\author{Antonio J. Di Scala, Naohiko Kasuya and Daniele Zuddas}
\address{\normalfont Antonio J. Di Scala, Dipartimento di Scienze Matematiche `G.L. Lagrange', Politecnico di Torino, Corso Duca degli Abruzzi 24, 10129, Torino, Italy.}
\email{antonio.discala@polito.it}
\address{\normalfont Naohiko Kasuya, Department of Social Informatics, Aoyama Gakuin University, 5-10-1 Fuchinobe, Chuo-ku, Sagamihara, Kanagawa 252-5258, Japan.}
\email{nkasuya@si.aoyama.ac.jp}
\address{\normalfont Daniele Zuddas, Korea Institute for Advanced Study, 85 Hoegiro, Dongdaemun-gu, Seoul 02455, Republic of Korea.}
\email{zuddas@kias.re.kr}

\date{}


\begin{document}

\begin{abstract}
We construct the first examples of non-K\"ahler complex structures on $\R^4$. These complex surfaces have some analogies with the complex structures constructed in early Fifties by Calabi and Eckmann on the products of two odd-dimensional spheres. However, our construction is quite different from that of Calabi and Eckmann.
\end{abstract}

\maketitle

\section{Introduction}
In the early Fifties, Calabi and Eckmann \cite{CE53} constructed an integrable complex structure on the Cartesian product of odd-dimensional spheres $M_{p, q} = S^{2p+1}\times S^{2q+1}$. These complex manifolds are nothing but complex tori for $p = q = 0$, while for $p \geq 1$ and $q = 0$ (or $p = 0$ and $q \geq 1$), they are Hopf manifolds \cite{Ho48}.
It is remarkable that for all non-zero $p$ and $q$, the manifolds $M_{p,q}$ have been the first examples of closed, simply connected complex manifolds which are not algebraic. Moreover, there is a holomorphic torus bundle $h_{p,q} \colon M_{p,q} \to \CP^p \times \CP^q$ given by the Hopf fibration on each factor.

By removing a point on each sphere and taking the product, we get an open subset $E_{p,q} \subset M_{p,q}$ which is diffeomorphic to $\R^{2p+2q+2}$.
If $p,q \geq 1$, most fibers of the bundle $h_{p,q}$ are contained in $E_{p,q}$. Thus, $E_{p,q}$ contains embedded holomorphic tori.
Therefore, it neither admits a K\"{a}hler metric, nor can be covered by a single holomorphic coordinate chart.
Calabi and Eckmann also proved that the only holomorphic functions on $E_{p,q}$ are the constants.

\begin{definition}
	A complex manifold $M$ is said to be of Calabi-Eckmann type if there exist a compact complex manifold $X$ of positive dimension, and a holomorphic immersion $k \colon X \to M$ which is null-homotopic as a continuous map. 
\end{definition}

It follows that a Calabi-Eckmann type complex manifold cannot be tamed by a symplectic form, and in particular it is not K\"ahler. As a consequence, Stein manifolds, complex algebraic manifolds, and open subsets of $\C^n$ with the induced complex structure, are not of Calabi-Eckmann type.

On the other hand, the manifolds $M_{p, q}$ and $E_{p, q}$ are of Calabi-Eckmann type for $p, q \geq 1$. Notice that $\R^4$ is not included in this list. Also notice that the Hopf manifolds $M_{p,0} = S^{2p+1} \times S^1$, $p\geq 1$, are not of Calabi-Eckmann type, since their universal cover is $\C^{p+1} - \{0\}$ (cf. Proposition~\ref{simpleprop/thm}).

The aim of this article is to construct Calabi-Eckmann type complex structures on $\R^4$. This represents a major improvement of a result of \cite{SKZ15}, where we constructed a {\sl not integrable} almost complex structure on $\R^4$ that contains embedded holomorphic tori and an immersed holomorphic sphere with one node. The methods used there have been inspired by previous work of the second author \cite{Kas14, Kas15} (see also \cite{DSV2010} and \cite{DZ11} for related results).

The following proposition is an immediate consequence of the definition.

\begin{proposition}\label{simpleprop/thm}
Let $M$ and $N$ be complex manifolds, with $M$ of Calabi-Eckmann type. Then, $N$ is of Calabi-Eckmann type if either
\begin{enumerate}
\item there is an immersion of $M$ into $N$, or
\item there is a covering map $p \colon N \to M$.
\end{enumerate}
\end{proposition}

Throughout this paper, we denote by $P$ the open subset of the plane defined by $$P = \{ (\rho_1, \rho_2) \in \R^2 \mid 0 < \rho_1 <1,  1<\rho_2 < \rho_1^{-1} \},$$ and we always assume $(\rho_1, \rho_2) \in P$.

We are now ready to state our main theorem.

\begin{theorem} \label{HF}
There is a family of Calabi-Eckmann type complex structures $\{J(\rho_1, \rho_2)\}$ on $\R^4$, parametrized by $(\rho_1, \rho_2) \in P$, and a surjective map $f\colon \R^4 \to \CP^1$ with only one critical point, such that:
\begin{enumerate}
\item $f$ is holomorphic with respect to $J(\rho_1, \rho_2)$ and the complex hessian at the critical point of $f$ is of maximal rank, for all $(\rho_1, \rho_2) \in P$;
\item the only holomorphic functions on $(\R^4, J(\rho_1, \rho_2))$ are the constants;
\item $J(\rho_1, \rho_2)$ depends smoothly on $(\rho_1, \rho_2) \in P$;
\item $(\R^4, J(\rho_1, \rho_2))$ is not biholomorphic to $(\R^4, J(\rho_1', \rho_2'))$ for any $(\rho_1, \rho_2) \neq (\rho_1', \rho_2')$;
\item the fibers of $f$ are either an immersed holomorphic sphere with one node, embedded holomorphic cylinders, or embedded holomorphic tori.
\end{enumerate}
\end{theorem}

We denote by $E(\rho_1,\rho_2)$ the complex manifold $(\R^4, J(\rho_1, \rho_2))$.

\begin{remark}
Property $(1)$ implies that $f$ can be locally expressed by $f(z_1, z_2) = z_1^2 + z_2^2$ in a neighborhood of the critical point, with respect to suitable local holomorphic coordinates in the complex structure $J(\rho_1, \rho_2)$. In other words, $f$ has a Lefschetz critical point. In fact, as we shall see below, $f$ is the restriction of an achiral Lefschetz fibration on $S^4$.
\end{remark}

We point out that, as far as we know, $J(\rho_1, \rho_2)$ are the first examples of non-K\"ahler complex structures on $\R^4$.

The following proposition gives a classification of the holomorphic curves of $E(\rho_1, \rho_2)$.

\begin{proposition}\label{classcurv/thm}
If $S$ is a closed Riemann surface, and $g \colon S \to E(\rho_1, \rho_2)$ is holomorphic, then either $g$ is constant or $g(S)$ is a closed fiber of $f$. It follows that the only compact holomorphic curves of $E(\rho_1, \rho_2)$ are the compact fibers of $f$, namely embedded holomorphic tori or the immersed holomorphic sphere.
\end{proposition}

The following is a corollary of Theorem~\ref{HF}.

\begin{corollary} \label{blowups}
The blowup $E(\rho_1,\rho_2) \cs m \overline{\CP^2}$ is a Calabi-Eckmann type complex manifold.
In particular, $m \overline{\CP^2} - \left\{\pt \right\}$ admits uncountably many non-K\"{a}hler complex structures, that are pairwise biholomorphically distinct. Moreover, the only holomorphic functions on the blowup $E(\rho_1,\rho_2) \cs m \overline{\CP^2}$ are the constants.
\end{corollary}

The paper is organized as follows.
In Section 2, we recall the construction of the Matsumoto-Fukaya torus fibration on $S^4$, which is a genus-1 Lefschetz fibration over $S^2$. This fibration plays a central role in the proof of Theorem~\ref{HF}.
As an application, we derive a certain decomposition of $\R^4$ in Proposition~\ref{R^4}.

In Section 3, we construct the complex structure $J(\rho_1, \rho_2)$ as well as the holomorphic map $f$, and we prove Theorem~\ref{HF} and Corollary~\ref{blowups}.

\section{The Matsumoto-Fukaya fibration on $S^4$}~\label{Matsumoto}
In the early Eighties, Yukio Matsumoto constructed a genus-1 achiral Lefschetz fibration $f \colon S^4 \to S^2$, having two critical points of opposite signs \cite{Mat82}. As it has been remarked by Matsumoto himself in the same article, Kenji Fukaya gave an important contribution in the understanding of this fibration. For this reason, in a private conversation Matsumoto suggested to us to call $f$ the {\sl Matsumoto-Fukaya fibration}, and we are glad to follow his suggestion.

Without going into details, $f$ can be defined as follows. Start with the Hopf fibration $h \colon S^3 \to S^2$, and take the suspension $\varSigma h \colon \varSigma S^3 \to \varSigma S^2$. There is a canonical smoothing $\varSigma S^n \cong S^{n+1}$, which makes the suspension $\varSigma h$ into a smooth map $\varSigma h \colon S^4 \to S^3$, see also \cite{SKZ15} for an explicit computation.

The composition $f' = h \circ \varSigma h$ is a torus fibration with two Lefschetz singularities, but the two critical points belong to the same fiber. Indeed, the following formula can be easily obtained \cite{SKZ15}
$$f'(z_1, z_2, x) = (4 z_1 \bar z_2 (|z_1|^2 - |z_2|^2 - i x \sqrt{2 - x^2}),\: 8 |z_1|^2 |z_2|^2 - 1),$$
where $S^{2n}$ is thought as the unit sphere in $\C^n \times \R$ defined by the equation $$|z_1|^2 + \cdots + |z_n|^2 + x^2 = 1.$$ In order to get two distinct singular fibers, we slightly perturb $f'$ and the result is the Matsumoto-Fukaya torus fibration $f \colon S^4 \to S^2$.
A description of this fibration is also given in \cite[Example 8.4.7]{GS99} in terms of a Kirby diagram, that is depicted in Figure~\ref{s4fib/fig}, where the framings are referred to the blackboard framing.

\begin{figure}[htb]
\centering\includegraphics{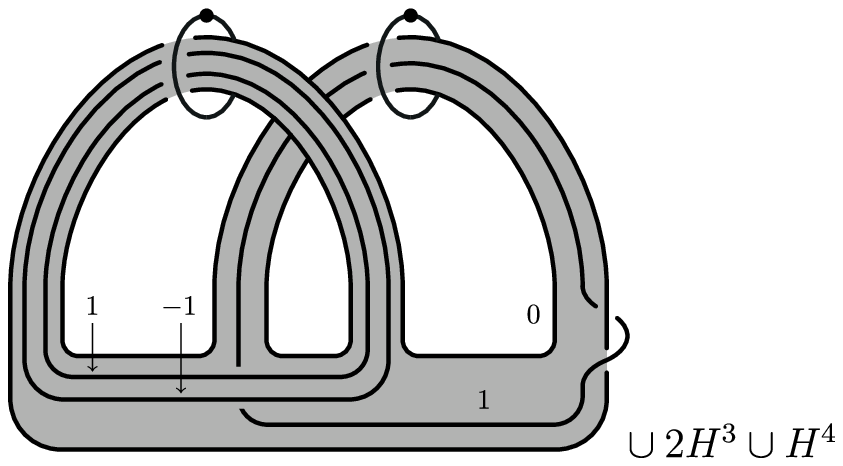}
\caption{The Matsumoto-Fukaya fibration on $S^4$.}
\label{s4fib/fig}
\end{figure}

We now explain this Kirby diagram and show how it can be derived. Let $a_1 \in S^2$ be the positive critical value of $f$, and let $a_2$ be the negative one.
Decompose the base space $S^2$ as the union of two disks $D_1$ and $D_2$ such that $a_j \in \Int D_j$, and put $N_j=f^{-1}(D_j)$. 
Then $N_j$ is a tubular neighborhood of $F_j = f^{-1}(a_j) \subset S^4$.

It follows that $\partial N_1$ and $\partial N_2$ are torus bundles over the circle, that are identified by a fiber-reversing diffeomorphism. So, $\partial N_1$ and $\partial N_2$ are essentially the same torus bundle, but with different orientations.

Let us consider the (achiral) Lefschetz fibration $f_j = f_{|N_j} \colon N_j \to D_j \cong B^2$, having only one critical point. It can be easily realized that the monodromy of $f_1$ is a right handed Dehn twist $\delta$ about an essential simple curve $c \subset T^2$ that can be identified with a meridian of the torus, while the monodromy of $f_2$ is given by $\delta^{-1}$. For generalities on Lefschetz fibrations and their monodromies, a good reference is \cite{GS99}. A description of the monodromy and the induced handlebody decomposition is given also in \cite{APZ2013}.

Therefore, $f_2$ can be identified with $f_1$ by reversing the orientation of the base disk, and keeping the same orientation on the fiber.

Since $f$ is built on the Hopf fibration, the monodromy of the latter reflects on the gluing diffeomorphism $\phi$ between $N_1$ and $N_2$. Namely, we have $$\partial N_1 = \frac{[0, 2\pi] \times T^2}{(0, \delta(x, y)) \sim (2\pi, x, y)},$$ where $(x, y)$ are the angular coordinates in $T^2 = S^1 \times S^1$, and the attaching diffeomorphism $\phi \colon \partial N_1 \to \partial N_2 = -\partial N_1$ is given by $\phi(t, x, y) = (t, x, y + t)$, which passes to the quotient. In other words, while running over $\partial D_1$, the fiber rotates along the longitude of $2\pi$ radians.

In Figure~\ref{s4fib/fig} the two 2-handles attached along parallel curves correspond to the two Lefschetz critical points, giving the corresponding vanishing cycles that are parallel to the curve $c$. The 2-handle with framing 0 attached along the boundary of the punctured torus is needed to close the fiber.

At this point, the Kirby diagram of Figure~\ref{s4fib2/fig} describes the fiber sum of $f_1$ and $f_2$ along an arc in $\partial B^2$.

\begin{figure}[htb]
\centering\includegraphics{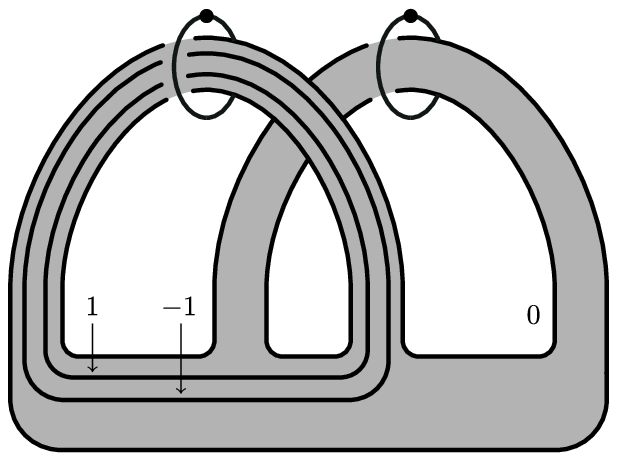}
\caption{The fiber sum of $f_1$ and $f_2$.}
\label{s4fib2/fig}
\end{figure}

In order to complete the fibration, we have to glue a trivial bundle $B^2 \times T^2$ by a fiber-preserving diffeomorphism given by $2\pi$ radians rotation in the longitudinal direction.

Considering $B^2$ as a 2-handle and taking the product with the standard handle decomposition of $T^2$ we get an extra 2-handle attached along a section, that follows the longitude. A simple computation shows that this 2-handle has framing one. Also, we get two 3-handles and a 4-handle.

Removing a neighborhood $X$ of the singular point of $F_2$ which is diffeomorphic to $B^4$, we obtain $\R^4$.
We define the subset $X$ of $N_2$ to be the standard model of the neighborhood of a negative Lefschetz singularity.
That is, $X$ is the total space of a singular annulus fibration over $D_2$ with one singular fiber and a left-handed Dehn twist as the monodromy.
It is well-known that $X$ is diffeomorphic to $B^4$, up to smoothing the corners.
Then $N_2 - \Int X$ is the total space of a trivial annulus bundle over $D_2$.

In other words, from the Kirby diagram of Figure~\ref{s4fib/fig} we are removing a 4-handle, the 3-handle that comes from the longitude of the torus, and the  2-handle with framing one that comes from the negative Lefschetz critical point, thus obtaining the diagram of Figure~\ref{r4fib/fig}. A simple computation shows that this represents $B^4$, taking into account that the 3-handle immediately cancels with the 0-framed 2-handle, as it results from its attaching map.

This Kirby diagram encodes the map $f_| \colon S^4 - \Int X \cong B^4 \to S^2$ as part of the Matsumoto-Fukaya fibration.

\begin{figure}[htb]
\centering\includegraphics{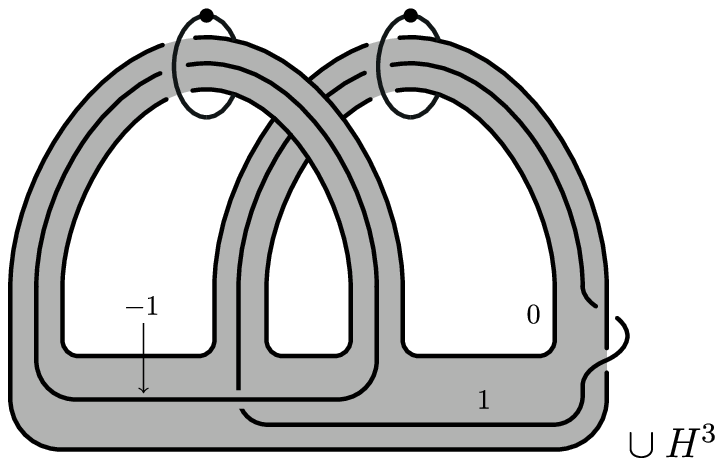}
\caption{The map $f$ on $B^4$.}
\label{r4fib/fig}
\end{figure}

The above considerations can be resumed in the following proposition, where $A = S^1 \times [0, 1]$.

\begin{proposition}~\label{R^4}
If we glue $B^2 \times A$ to $N_1$ along $S^1 \times A$ so that for each $t\in \partial B^2 = -\partial D^2_1 \cong S^1$,
the annulus $\{t\} \times A$ embeds in $f^{-1}(t)\cong T^2$ as a thickened meridian,
and it rotates in the longitude direction once when $t\in S^1$ rotates once, then the resulting manifold is diffeomorphic to $B^4$, and so the interior is diffeomorphic to $\R^4$.
\end{proposition}

\section{Construction of the complex structures}
In the following we make use of the notations
$\Delta(r_0, r_1) = \{z \in \C \mid r_0 < |z| < r_1 \}$, $\Delta(r_1) = \{z \in \C \mid \allowbreak |z| < r_1 \}$, for $0 \leq r_0 < r_1$. We put also $\C^* = \C - \{0\}$.
For a given $w \in \C^*$ such that $|w| < 1$, we consider the smooth elliptic curve $T^2_w=\C^{\ast }\slash \Z \cong T^2$, where the action is given by $n\cdot z = w^n z$. 
We call the curve $\mu = \{|z|=1\} \subset T^2_w$ the meridian and the curve $\lambda = \{\arg z =0 \} \subset T^2_w$ the longitude of the torus.
These meridian and longitude can be identfied with those of the previous section.

\begin{proof}[Proof of Theorem~\ref{HF}]
We begin with the construction of the complex structure $J(\rho_1, \rho_2)$. 
Since $N_1$ is the total space of a positive genus-$1$ Lefschetz fibration over the $2$-disk with one singular fiber,
there exists a complex structure such that the fibration $$f_{|N_1} \colon N_1\to D_1$$ is holomorphic.
Indeed, we consider a holomorphic elliptic fibration over $S^2$, and
we take a tubular neighborhood of a singular fiber which is fiberwise diffeomorphic with $N_1$ (see also \cite{GS99}).
According to \cite{Kod63}, we can give a more explicit model of $\Int N_1$ by the Weierstrass curves.
For $0<\rho_1<1$, we consider the complex submanifold
$$S = \left\{([z_0 {:} z_1 {:} z_2], \tau )\mid z_1^2z_2-4z_0^3-z_0^2z_2+g_2(\tau )z_0z_2^2+g_3(\tau )z_2^3=0 \right\}\subset \CP^2\times \Delta (\rho_1),$$
where
\begin{align*}
g_2(\tau ) &= 20\sum _{n=1}^{\infty } (1-\tau ^n)^{-1}n^3 {\tau }^n\\
g_3(\tau ) &= \frac{1}{3}\sum _{n=1}^{\infty }(1-\tau ^n)^{-1}(7n^5+5n^3) {\tau }^n.
\end{align*}

For each $\tau \in \Delta (\rho_1)$, the fiber $\left\{z_1^2z_2-4z_0^3-z_0^2z_2+g_2(\tau )z_0z_2^2+g_3(\tau )z_2^3=0 \right\}\subset \CP^2$ is an elliptic curve and it is singular only for $\tau =0$. Using the coordinates $(x,y)=(z_0/z_2, z_1/z_2)$, the singular elliptic fiber is defined by the equation
$$y^2-4x^3-x^2=0$$ and it has an ordinary double point at $x=y=0$.
Hence, the canonical projection $$\pi \colon \CP^2 \times \Delta (\rho_1)\to \Delta (\rho_1)$$ restricts to a holomorphic map $\pi_{|S} \colon S \to \Delta (\rho_1)$,
which is a genus-$1$ holomorphic Lefschetz fibration over the $2$-disk with one singular fiber.
Thus, the complex manifold $S$ is a complex model of $\Int N_1$.

Now, we consider the quotient $(\C^{\ast} \times \Delta (0,\rho_1)) / \Z$, where for any $n\in \Z$, the action is given by
$$n \cdot (z, w) = (z w^n, w).$$
This elliptic fibration extends over $\Delta (\rho_1)$. Let us denote the completion by $W$.
Kodaira gave an explicit biholomorphism between $W$ and $S$ (see \cite{Kod63}, pp. 597--599).
So, in the following we shall consider $W$ as the model of $\Int N_1$, instead of the Weierstrass model $S$.

We fix a holomorphic atlas on the Riemann sphere given by two open disks $D_1' \supset D_1$ and $D_2' \supset D_2$. The disk $D_1'$ is biholomorphic with $\Delta(\rho_1)$ and $D_2'$ is biholomorphic with $\Delta(\rho_0^{-1})$, where $\rho_0 \in (0, \rho_1)$ is arbitrarily chosen, and the transition function $\psi \colon \Delta(\rho_0, \rho_1) \to \Delta(\rho_1^{-1}, \rho_0^{-1})$ is given by $\psi(z) = z^{-1}$.

We define the complex structure on the topologically trivial annulus bundle $N_2 - X$, considered over $\Delta(\rho_0^{-1}) \cong D_2'$, by the product structure 
$\Delta(1,\rho_2) \times \Delta(\rho_0^{-1})$. 

Next, we want to glue $W \cong \Int N_1$ with $\Delta(1,\rho_2) \times \Delta(\rho_0^{-1}) \cong \Int N_2 - X$ analytically along $\Delta(1,\rho_2) \times \Delta(\rho_1^{-1}, \rho_0^{-1})\allowbreak \subset \Delta(1,\rho_2) \times \Delta(\rho_0^{-1})$, so that the attaching map is isotopic to that of the Matsumoto-Fukaya fibration, implying that the resulting manifold is diffeomorphic to $\R^4$. 
In order to do this, we need to choose an attaching region in $W$ which is biholomorphic to the product $\Delta(1,\rho_2) \times \Delta(\rho_1^{-1}, \rho_0^{-1})$. 
In the following argument, we show how to take such a region in $W$.

We begin with a multi-valued holomorphic function $\phi \colon \Delta(\rho_0, \rho_1) \to \C^*$ such that multiplication by $w^n$ for all $n \in \Z$, determines a transitive $\Z$-action on the set of the branches of $\phi$. In other words, given any branch $\phi_0$ of $\phi$, all the other branches are of the form $w^n \phi_0(w)$ for an arbitrary $n \in \Z$.

For example, as it can be easily verified, we can take
\begin{eqnarray*}
\phi(w)=\exp \left(\frac{1}{4\pi i}(\log w)^2 - \frac{1}{2} \log w \right),
\end{eqnarray*}
where the two logarithms are taken simultaneously with all of their possible branches.

Next, consider the open subset $Y\subset \C ^{\ast} \times \Delta (0, \rho_1)$ defined by
\begin{eqnarray*}
Y=\left\{(z, w)\in \C ^{\ast }\times \Delta \big( \rho_0,\; \rho_1 \big) \mid  z \phi (w)^{-1} \in \Delta (1, \rho_2) \text{ for some value of } \phi(w) \right\}.
\end{eqnarray*} 
Then, $Y$ can be parameterized by the multi-valued local holomorphic immersion $$\Phi \colon \Delta (1, \rho_2) \times \Delta(\rho_0, \rho_1) \to \C^* \times \Delta(0, \rho_1)$$ defined by $\Phi(z, w) = (z \phi(w), w)$. Notice that $Y$ is invariant under the action of $\Z$ on $\C^* \times \Delta(0, \rho_1)$.

It follows that the composition of $\Phi$ with the quotient map $$\pi \colon \C^* \times \Delta(0, \rho_1) \to (\C^* \times \Delta(0, \rho_1)) / \Z \subset W$$ is a single-valued holomorphic embedding, and we denote by $V$ the image of $Y$ in $W$.

Also notice that $f_{|V} \colon  V\to \Delta(\rho_0, \rho_1)$ is a holomorphic annulus bundle, and $\Phi$ determines a trivialization of this bundle.

Let $j\colon \Delta (1, \rho_2)\times \Delta (\rho_0, \rho_1) \to \Delta(1, \rho_2) \times \Delta(\rho_1^{-1}, \rho_0^{-1})$ be the biholomorphism defined by $j(z, w) = (z, w^{-1})$.
We use $\pi \circ \Phi \circ j^{-1} \colon \Delta(1, \rho_2) \times \Delta(\rho_1^{-1}, \rho_0^{-1}) \to V$ as the attaching biholomorphism for making the union
$$E(\rho_1,\rho_2) = (\Delta (1, \rho_2)\times \Delta ({\rho_0}^{-1})) \cup_V W,$$ which therefore is a complex manifold. We denote by $J(\rho_1, \rho_2)$ the complex structure of $E(\rho_1, \rho_2)$.

In order to identify the topology of $E(\rho_1, \rho_2)$,
we consider how the annulus fiber of $V=\Delta (1, \rho_2)\times \Delta(\rho_0, \rho_1)$ looks inside the toric fiber of $W \cong \Int N_1$. Let $w \in \Delta(\rho_0, \rho_1)$ be a complex number of some fixed modulus, and of arbitrary argument $\arg w$.
When $\arg(w)$ varies from $0$ to $2\pi$, a point $z$ of the fiber $\Delta(1, \rho_2)$ moves accordingly with the $\Z$-action on the branches of $\phi$, which encodes the attaching biholomorphism. Namely, $z$ goes to $w z$, that is in the next fundamental domain of the $\Z$-action on $\C^*$ that gives the torus $T^2_w = \C^* / \Z$.
This means that when $w$ rotates once in the argument direction, the annulus $\Delta (1, \rho_2)$ rotates once in the longitude direction of the fiber torus $T^2_w$, and this coincides with the attaching map of the Matsumoto-Fukaya fibration restricted on $S^4 - X$.
This implies that $E(\rho_1, \rho_2)$ is diffeomorphic to $\R^4$.

Observe that the number $\rho_0$ is auxiliary. 
Indeed, its role is just to provide an open subset $V$ of $\Int N_1$ which is necessary for the analytical gluing. 
Namely, $\rho_0$ only determines the size of the attaching region of the two complex manifolds. 
Hence, $E(\rho_1, \rho_2)$ does not depend on $\rho_0$ up to biholomorphisms.

Moreover, we can define a surjective holomorphic map $$f_{\rho_1, \rho_2} \colon E(\rho_1, \rho_2) \to \CP^1$$
by $f_{\rho_1, \rho_2}(z, w) = w$ if $(z, w) \in W$ and $f_{\rho_1, \rho_2}(z, w) = w^{-1}$ if $(z, w)\in \Delta (1, \rho_2)\times \Delta ({\rho_0}^{-1})$.

Notice that $f_{\rho_1, \rho_2}$ is isotopic to the restriction of the Matsumoto-Fukaya fibration $f$ (up to the diffeomorphism $E(\rho_1, \rho_2) \cong \R^4$). Actually, the isotopy can be chosen to be smooth with respect to the parameters $(\rho_1, \rho_2)$, and up to this isotopy we can assume that $f_{|\R^4}$ itself is a holomorphic fibration with respect to all of the complex structures $J(\rho_1, \rho_2)$. Moreover, the complex hessian of $f$ at the critical point, computed with respect to $J(\rho_1, \rho_2)$, is of maximal rank because this complex structure on $W$ coincides with that of the Weierstrass curves model.

Since $f$ has holomorphic tori as fibers, it follows that $E(\rho_1, \rho_2)$ is of Calabi-Eckmann type. Moreover, statements $(1)$, $(3)$ and $(5)$ of the theorem are implicit in the construction.

The non-existence of non-constant holomorphic functions also follows easily, since there is an open subset of $E(\rho_1,\rho_2)$, namely $W$, which is foliated by compact holomorphic curves, hence a holomorphic function $g \colon E(\rho_1,\rho_2) \to \C$ must be constant on these fibers. Therefore, the differential of $g$ is zero along those compact fibers. Since $dg$ is a holomorphic 1-form, it follows that $d g$ is zero even along the annulus fibers of $f$, and this implies that $g$ is constant on the fibers of $f$. Thus, $g$ factorizes by the fibration $f$. Namely, there is a holomorphic function $g' \colon \CP^1 \to \C$ such that $g = g' \circ f$. Since $g'$ is constant, $g$ is also constant. This proves statement $(2)$.

Now we give the proof of statement $(4)$, which is based on Proposition~\ref{classcurv/thm} that will be proved at the end of this section. 
By Proposition~\ref{classcurv/thm}, the union of all compact holomorphic curves of $E(\rho_1,\rho_2)$ is the open subset $W$. Analogously, we denote by $W'$ the union of the compact holomorphic curves of $E(\rho_1',\rho_2')$.

Suppose that there is a biholomorphism $g\colon E(\rho_1,\rho_2)\to E(\rho_1',\rho_2')$. We want to show that $(\rho_1, \rho_2) = (\rho_1', \rho_2')$.

The above discussion implies that $g$ decomposes into two biholomorphisms 
$g_{|} \colon W \to W'$ and
$g_{|} \colon \Delta (1, \rho_2)\times \Delta ({\rho_1}^{-1}) \to \Delta (1, \rho_2')\times \Delta ({\rho_1'}^{-1})$. Moreover, $g$ is fiber-preserving on $W$, and so $g_{|W}$ passes to the quotient, giving a biholomorphism $g' \colon \Delta(\rho_1) \to \Delta(\rho_1')$. By analyticity, $g$ must be fiber-preserving also on $\Delta (1, \rho_2)\times \Delta ({\rho_1}^{-1})$. This immediately gives $\rho_2 = \rho_2'$, because of the well-known holomorphic classification of complex annuli.

The torus $T^2_w$ that corresponds to the complex number $w \in \Delta(0, \rho_1)$, is isomorphic to a complex torus of the form $\C / (\Z \oplus \Z v)$, where $$v = \frac{1}{2\pi i} \log w = \frac{1}{2\pi} \arg w - \frac{i}{2\pi} \log |w|,$$ and $\arg w \in [0, 2\pi)$.

By the classification of complex non-singular elliptic curves, $g'$ must be the identity because $T^2_w$ is isomorphic to $g(T^2_w) = T^2_{g'(w)}$ for all $w \in \Delta(0, \rho_1)$. Therefore, we obtain $\rho_1 = \rho_1'$.
\end{proof}

\begin{proof}[Proof of Proposition \ref{classcurv/thm}]
It is sufficient to show that $f \circ g$ is constant. In fact, $f \circ g \colon S \to \CP^1$ is homotopic to a constant, since it factorizes through the contractible space $E(\rho_1, \rho_2)$. Therefore, $f \circ g$ is of degree zero. 
Since it is a holomorphic map between compact Riemann surfaces, it must be constant.
\end{proof}


Finally, we prove Corollary~\ref{blowups}. 

\begin{proof}[Proof of Corollary~\ref{blowups}]
Since the blowup affects $E(\rho_1, \rho_2)$ only at finitely many points, after blowing up, there are still embedded holomorphic tori which are homotopically trivial. Then, $E(\rho_1, \rho_2) \cs m \overline{\CP^2}$ is of Calabi-Eckmann type.

Moreover, if $E(\rho_1,\rho_2) \cs m \overline {\CP^2}$ and $E(\rho_1',\rho_2') \cs m \overline {\CP^2}$ are biholomorphic, 
then it follows that $\rho_1=\rho_1'$ and $\rho_2=\rho_2'$ by the same argument in the proof of Theorem~\ref{HF} $(4)$. 
Since $E(\rho_1,\rho_2) \cs m \overline {\CP^2}$ is diffeomorphic to $m \overline {\CP^2}- \left\{p \right\}$, 
there are uncountably many distinct non-K\"{a}hler complex structures on $m \overline {\CP^2}- \left\{p \right\}$.

Finally, if $h$ is a holomorphic function on $E(\rho_1, \rho_2) \cs m \overline{\CP^2}$, it must be constant on the exceptional spheres, and so it factorizes through the blowup map $$\sigma \colon E(\rho_1, \rho_2) \cs m \overline{\CP^2} \to E(\rho_1, \rho_2).$$ Hence, $h$ is constant. 
This completes the proof. 
\end{proof}

\begin{remark}
Corollary~\ref{blowups} holds even for $m = \infty$, and the proof is essentially the same. By making the points for the blowups to vary, we get even more pairwise inequivalent complex structures on $E(\rho_1, \rho_2) \cs m\overline{\CP^2}$.
\end{remark}

\section*{Acknowledgements}
The authors are grateful to Professors Yukio Matsumoto, Ryushi Goto, and Ichiro Enoki.
Yukio Matsumoto gave us an important suggestion about Theorem~\ref{HF}.
Ryushi Goto and Ichiro Enoki gave us helpful advice for the model of the holomorphic elliptic fibration in the proof of Theorem~\ref{HF}.
Part of this article has been written when the second and the third authors were visiting the Department of Mathematical Sciences of Durham University, UK. So, they are grateful to Durham University for hospitality. The third author also thanks the Grey College of Durham University for hospitality during his stay at Durham. We are also thankful to Professor Wilhelm Klingenberg for helpful conversations and for having invited us to Durham University.

Antonio J. Di Scala is member of the Italian PRIN 2010-2011 ``Variet\`a reali e complesse: geometria, topologia e analisi armonica" and member of GNSAGA of INdAM. Daniele Zuddas is member of GNSAGA of INdAM.

\end{document}